\documentclass[11pt]{amsart}
\usepackage{amscd,amssymb,epsfig,mathtools,setspace,subfig}%,accentset}
\usepackage[usenames,dvipsnames]{color}
\usepackage{url}
\calclayout

\numberwithin{equation}{section} 

\newcommand{\ud}{\,d} 
\newcommand{\Sym}{\mathbb{S}} 
\newcommand{\R}{\mathbb{R}}

\DeclareMathOperator{\dist}{dist}

\newcommand{\tir}[1]{\ensuremath{\overline {#1}}} 
\newtheorem{thm}{Theorem}[section] 
 
\newtheorem{lemma}[thm]{Lemma} 
 
\newtheorem{prop}[thm]{Proposition} 
 
\newtheorem{defn}[thm]{Definition} 
 
\newtheorem{rem}[thm]{Remark}

\def\whsq{\vbox to 5.8pt 
{\offinterlineskip\hrule 
\hbox to 5.8pt{\vrule height 
5.1pt\hss\vrule height 5.1pt}\hrule}}

\def\<{\langle} 
\def\>{\rangle} 
\def\PP{{\mathop{{\rm I}\kern-.2em{\rm P}}\nolimits}} 
\def\FF{{\mathop{{\rm I}\kern-.2em{\rm F}}\nolimits}}   
\def\ZZ{{\mathop{{\rm I}\kern-.2em{\rm Z}}\nolimits}} 
 
\newcommand{\edit}[1]{{\color{black} #1}}
\newcommand{\st}[1]{}
 \graphicspath{{}{../Figures/}{}}    
\newlength{\sidemargin} 
\begin{document}
\title[]{
A variational method  for computing numerical solutions of the Monge-Amp\`ere equation
}

%\thanks{Not for dissemination, May 16, 2014 }

\thanks{%Gerard Awanou was supported in part by NSF DMS grant No 1319640. 
This research was supported in part by NSF grant DMS-1319640 to Gerard Awanou and NSF grant DMS-0931642 to the Mathematical Biosciences Institute.
}
\author{Gerard Awanou}
\address{Department of Mathematics, Statistics, and Computer Science, M/C 249.
University of Illinois at Chicago, 
Chicago, IL 60607-7045, USA}
\email{awanou@uic.edu}  
\urladdr{http://www.math.uic.edu/\~{}awanou}

\author{Leopold Matamba Messi}
\address{Mathematical Biosciences Institute,
The Ohio State University,
Columbus, OH 43210, USA}
\email{matambamessi.1@mbi.osu.edu}
\urladdr{http://people.mbi.ohio-state.edu}

\maketitle

\begin{abstract} 
We present a numerical method for solving the Monge-Amp\`ere equation based on 
the characterization of the solution of the Dirichlet problem as the minimizer 
of a convex functional of the gradient and under convexity and nonlinear 
constraints. When the equation is discretized with a certain monotone scheme, we 
prove that the unique minimizer of the discrete problem solves the finite 
difference equation. For the numerical results we use both the standard finite 
difference discretization and the monotone scheme. Results with standard tests confirm that the 
numerical approximations converge to the Aleksandrov solution.

\end{abstract}

\section*{Introduction}
We are interested in the numerical resolution of the Monge-Amp\`ere 
equation
\begin{equation}
\det  \nabla^2 u   = f  \ \text{in} \ \Omega, \quad u=g \ \text{on} \ \partial \Omega, 
\label{m1}
\end{equation}
where the domain $\Omega \subset \R^n$ is assumed to be bounded and convex with  
boundary $\partial \Omega$.  For a smooth function $u$, $\nabla^2 u$ is the Hessian 
of $u$ with entries  $(\partial^2 u) / (\partial x_i \partial x_j), i,j=1,\ldots, 
n$. In general the expression $\det \nabla^2 u$ should be interpreted in the sense of 
Aleksandrov.  The functions $f$ and $g$ are given  with $f\geq0$. We make the 
assumption that $f \in L^1(\Omega)$ and $g \in C^{0,1}(\partial \Omega)$ can be 
extended to a function $\tilde{g} \in C(\tir{\Omega})$ which is convex on 
$\Omega$.

We present a numerical method  based on a remark of P.L. Lions \cite{Lions86} 
which asserts that the unique $C^{0,1}(\tir{\Omega})$ solution of \eqref{m1} in 
the sense of Aleksandrov is the unique minimizer of 
$$
J(w) = \int_{\Omega} \Phi(\nabla w) \, dx ,
$$
over the convex set \st{for a nonnegative strictly convex function 
$\Phi$ satisfying Assumptions {\eqref{assumptions-phi}} below and over the 
convex set}
\begin{align}\label{s}
% \begin{split}
S = \big\{ v \in C(\tir{\Omega}),~ v=g \text{ on } \partial \Omega,
\, v \ \text{convex on} \, \Omega, \enskip (\det  \nabla^2 v)^{\frac{1}{n}}  \geq  
f^{\frac{1}{n}}    \ \text{in} \  \Omega \big\},
% \end{split}
\end{align}
{where $\Phi$ is a nonnegative strictly convex function satisfying 
\eqref{assumptions-phi}.}

Our strategy consists in \st{minimizing} {reproducing} the above principle 
by solving discrete versions of the convex program
\begin{equation}
 \arg\min_{v\in S} J(v).
\end{equation}

We consider in this paper two notions of discrete convexity: the first one  
which we will refer to as local discrete convexity requires a certain discrete 
Hessian to be positive. The second one, recently used in \cite{Oberman2010a}, 
will be referred to as wide stencil convexity and requires to enforce the 
convexity of the mesh function approximately. 

{Given a discretization of the functional $J$, we consider two 
possible discrete counterparts of the convex set $S$ corresponding to 
different approximations of $\det \nabla^2u$.} If the determinant operator 
is discretized with the monotone scheme of \cite{Oberman2010a}, and if one uses wide stencil convexity, we prove that 
the unique minimizer of the corresponding  discrete optimization problem solves the 
{associated} finite difference version of \eqref{m1}. However, uniqueness 
of the solution of the latter finite difference problem remains an open 
question. We present numerical evidence of convergence  
to the Aleksandrov solution for a test case for which solving directly the nonlinear finite difference equations does not 
give satisfactory results. We do not address in this paper the convergence of the discretization proposed in \cite{Oberman2010a}
to the Aleksandrov solution and refer to \cite{AwanouAwi2}.

{We also consider the discrete version of the convex set $S$ obtained 
from standard finite difference approximations of $\det \nabla^2v$ and local discrete convexity.  We prove the 
existence of the solution of the resulting discrete convex program}
\st{of the standard finite difference discretization}under the assumption that a 
solution of the standard finite difference equations exists. The existence of a  
solution to the finite difference equations in this case is obtained in  
\cite{Awanou-k-Hessian12} under the assumption that
\eqref{m1} has a smooth solution. Furthermore, using a weak convergence result proved in \cite{AwanouAwi2}, we prove convergence of minimizers 
when \eqref{m1} has a smooth solution.
We present numerical evidence of convergence  
to the Aleksandrov solution. Our results can be combined with the approach in 
\cite{Awanou-Std-fd} to give a convergence result for Aleksandrov solutions.

The unknown $u$ in \eqref{m1} is a convex function which may not be smooth even 
if the data are smooth. Approximating the appropriate weak solution and 
preserving numerically the convexity property had posed great challenges for the 
numerical resolution of \eqref{m1} in the context of standard discretizations. 
In this paper, we take a direct approach by including explicitly the convexity 
constraints in an optimization framework. The notion of viscosity solution and 
that of Aleksandrov solution are the best known notions of weak solution for 
\eqref{m1}. They are equivalent for $f>0$ and continuous \cite{Gutierrez2001}. 
We refer to \cite{Rauch77} and the references therein for the precise definition 
of { the concept of  Aleksandrov solution for \eqref{m1}}.

The convexity of the set $S$ in \eqref{s} follows from Minkowski determinant  
theorem. See for example \cite[Theorem G, p. 205]{Wayne73} for smooth functions 
and \cite[Proposition 3.3]{Rauch77} for a procedure for the general case using 
approximation by smooth functions. We recall that for a convex function  $v$ on 
a convex domain $\Omega$, Rademeister theorem \cite[Theorem 19 p. 13]{Shor98} 
states that the set of points at which $v$ fails to be differentiable has zero 
Lebesgue measure and that its gradient is continuous on the set of points where 
it exists. Thus the functional $J$ is well defined on $S$. In \eqref{s}, $\det  
\nabla^2 v$ denotes the Monge-Amp\`ere measure associated to the convex function 
 $v$,  \cite{Gutierrez2001}. We use the notation $C^{0,1}(\tir{\Omega})$ for the 
set of Lipschitz continuous functions on $\tir{\Omega}$. % and we recall that a continuous convex function is in $C^{0,1}(\tir{\Omega})$. 
Under our assumptions 
on $f$ and $g$, it can be shown that the Aleksandrov  solution of \eqref{m1} is 
the maximal element of $S$ when the domain is convex and not necessarily 
strictly convex \cite{Hartenstine2006}. Then the arguments in \cite{Lions86} 
extend to the case where $\Omega$ is assumed to be only convex. 

%We recall that the operator $\det 
%\nabla^2 u$ can be discretized using standard discretizations or using a monotone 
%scheme as in \cite{Oberman2010a}. 

\subsection*{Relation with a method of Dean and Glowinski}

The optimization approach we take has some similarity with an optimization  
approach proposed by Dean and Glowinski \cite{Dean2006}. See also  
\cite{Glowinski2014}. They proposed in \cite{Dean2006} to solve \eqref{m1} by 
minimizing
$$
L_1(w) = \int_{\Omega} (\Delta w)^2\,dx,
$$
over
$$
E=\{\, v \in H^2(\Omega), \, v=g \ \text{on} \ \partial \Omega, \det  
\nabla^2 v = f   \, \text{in} \  \Omega \}.
$$
An equivalent mixed formulation was used with additional unknown $q=\nabla^2 u$  
and then solved by an augmented Lagrangian algorithm. In the mixed formulation, 
the functional $L(w,q) = \int_{\Omega} |\Delta w|^2$ is minimized over the set
$$
E'=\{ (v, q) \in H^2(\Omega) \times L^2(\Omega,\Sym),   v=g \ \text{on}  \ 
\partial \Omega, q=\nabla^2 v, \det  q = f\},
$$
where $\Sym$ is the space of $2 \times 2$ symmetric matrices and  
$L^2(\Omega,\Sym)$ is the space of symmetric matrix fields with components in 
$L^2(\Omega)$.  In the case $f$ is unbounded the results were not satisfactory.  
As remarked in \cite{Brenner2010b}, the convergence of their method, even for 
smooth solutions, is still an open problem. 

It will be seen that if one replaces the functional $L_1$ by
$$
L_2(w) = \int_{\Omega} |\nabla w|^2~dx,
$$
and the set $E$ by the convex set $S$ defined in \eqref{s},
one obtains a method of the type discussed in this paper and the difficulties 
indicated for the method of Dean and Glowinski disappear. One of the main 
features of the method we propose is that the convexity constraint in \eqref{s} 
has been carefully taken into account. 

\subsection*{Organization of the paper}

The paper is organized as follows: in the next section, we introduce some 
notation and analyze the discrete optimization problem with a standard  
discretization of $\det \nabla^2 u$ and local discrete convexity.  We then prove 
that when the equation is discretized with the monotone scheme of  
\cite{Oberman2010a}, the unique minimizer of the discrete problem solves the 
finite difference equation. Section \ref{sec:num} is devoted to numerical  
results. % with the standard finite difference discretization. 

\section{Discrete optimization problem} \label{sec:disc}
 
We start with some notations needed to state the discrete optimization problem. 
All the functions we consider take finite values. We make the usual convention 
of denoting constants by the letter $C$. We make the following assumptions on 
the nonnegative convex function $\Phi$
\begin{subequations} \label{assumptions-phi}
\begin{equation}
\Phi \in C^2(\R^n) 
\end{equation}
\begin{equation}
|\nabla\Phi(p)| \leq C + C |p|
\end{equation}
\begin{equation}
\exists \, \nu > 0 \text{ such that }   \langle \nabla\Phi(p) 
- \nabla\Phi(q), p- q\rangle \geq \nu |p-q|^2, \quad \forall p, q \in \R^n
\end{equation}
\end{subequations}
where $\langle,\rangle$ denotes the Euclidean inner product in $\R^n$ and $|.|$ 
denotes the associated norm. Without loss of generality, we will assume, for the 
analysis  in this paper, that $\Phi$ is strictly convex. For example, one may 
take 
$\Phi(p) =  |p|^2 $. The results in \cite{Lions86} hold for more general
convex functions $\Phi$ and we also give numerical results for functions $\Phi$ 
not satisfying the above assumptions. 

\subsection{Notation and preliminaries}
Most of the notation is taken from \cite{Aguilera2008}. Let $\Omega=(0,1)^n,  h 
=1/N, N \in \mathbb{N}$ and let $\mathcal{M}_h$ denote the mesh which consists 
of points $x=h z  \in \R^n \cap \tir{\Omega}, z \in \mathbb{Z}^n$. Put
$$
\mathcal{M}_h^{\circ} = \mathcal{M}_h \cap \Omega, \quad \partial  \mathcal{M}_h 
= \mathcal{M}_h \cap  \partial \Omega,
$$ 
and denote by $\mathcal{U}_h$ the set of real valued functions defined on 
$\mathcal{M}_h$. Without loss of generality, for a function $v$ defined on 
$\tir{\Omega}$, we use the notation $v$ for the restriction of $v$ to  
$\mathcal{M}_h$.

\begin{defn}
If $u$ is defined on $\Omega$ and $u_h \in \mathcal{U}_h$, we say that $u_h$  
converges uniformly to $u$ on a compact subset $K$ of $\Omega$ if and only if
$$
\max_{x \in  \mathcal{M}_h^{\circ} \cap K} |u_h(x) - u(x)| \to 0,\text{ as }  
h \to 0.
$$
\end{defn}

We denote by $e_i, i=1,\ldots,n$ the $i$-th {coordinate} unit 
vector of $\R^n$ and consider first order difference operators defined on 
$\mathcal{M}_h^{\circ}$ by 
\begin{align*}
\partial^i_{+} v_h(x) & \coloneqq \frac{ v_h(x+he_i)-v_h(x) } {h} \\
\partial^i_{-} v_h(x) & \coloneqq \frac{ v_h(x)-v_h(x-he_i) } {h} \\
\partial^i_h v_h(x) & \coloneqq \frac{ v_h(x+he_i)-v_h(x-he_i) }{2 h}. 
\end{align*}
{We consider the following second-order difference operators on mesh functions} 
%are defined as follows.
\begin{align}
\partial^i_{+} \partial^i_{-} v_h(x) & = \frac{v_h(x+he_i)- 2 v_h(x)+ 
v_h(x-he_i)}{h^2}, \label{second-disc1}
\end{align}
and
\begin{align}
\begin{split}
\partial^i \partial^j v_h(x) & = \frac{1}{4 h^2} \bigg\{v_h(x+he^i+h e^j)+ 
v_h(x-he^i-h e^j) \\
& \qquad \qquad \qquad -v_h(x+he^i-h e^j)-v_h(x-he^i+ he^j)\bigg\}, i \neq j.  
\label{second-disc2}
\end{split}
\end{align}
We use the notation $A=(a_{i j})_{1\le i,j \le n}$ to denote the matrix 
$A$  with entries $a_{ij}$. {The Hessian of a mesh function $v_h$ is 
the discrete matrix field} $H[v_h]=\big(H[v_h]_{ij}\big)_{1\le i,j\le n}$ with 
entries {the mesh functions}
\begin{align*}
H[v_h]_{ij}(x) = \begin{cases}
                  \partial^i_{+} \partial^i_{-} v_h(x) & \text{ if } i=j,\\[2mm]
                  \partial^i \partial^j v_h(x) & \text{ otherwise.}
                 \end{cases}
\end{align*}

{
\begin{defn} We say that a mesh function $v_h$ is (locally) discrete convex 
if $H[v_h](x) $ is positive semidefinite for all $x \in \mathcal{M}_h^{\circ}$. 
The mesh function $v_h$ is (locally) strictly discrete convex if $H[v_h](x) $ 
is positive definite for all $x \in \mathcal{M}_h^{\circ}$.
\end{defn}
}

{We endow the space of mesh functions defined on $\mathcal{M}_h$ with the 
following norms and semi-norms.}
\begin{align*}
|v_h|_{0,\infty} & = \max_{x \in \mathcal{M}_h^{\circ}} |v_h(x)|, \\
|v_h|_{1,\infty} & = \max \{ \, |\partial^i_{+} v_h|_{0,\infty}, i=1,\ldots,n \, 
\}, \\
 |v_h|_{2,\infty} & = \max \{ \,  |\partial^i_{+} \partial^i_{-}  
v_h|_{1,\infty},  |\partial^i \partial^j v_h|_{1,\infty},  i,j=1,\ldots,n \,\}, 
\\
||v_h||_{2,\infty} & =  \max \{ \,   |v_h|_{0,\infty},   |v_h|_{1,\infty},   
|v_h|_{2,\infty} \, \}.
\end{align*}

Analogues of the Sobolev norms and semi-norms can be defined on $\mathcal{M}_h$. 
For $v_h \in \mathcal{U}_h$ we define
\begin{align*}
||v_h||_{0} = \bigg( h^n \sum_{x \in \mathcal{M}_h^{\circ}} v_h(x)^2  \bigg)^{ \frac{1}{2} },
\end{align*}
and
\begin{align*}
|v_h|_{1}  = \bigg( \sum_{i=1}^n ||  \partial^i_{+}  v_h||^2_{0}\bigg)^{\frac{1}{2}}.
\end{align*}

We have the discrete Poincare's inequality, see for example \cite[Lemma 3.1]{Chung97} %or \cite[Proposition 3.3]{Oleuinik92} %mod 1
\begin{lemma} \label{poincare}
There exists a constant $C >0$ independent of $h$ such that %if $|v_h|_{1,h} < \infty$, 
$$
|v_h|_{1}  \geq C ||v_h||_{0}, % \forall v_h \in  
$$
for $v_h=0$ on $\partial \mathcal{M}_h$.
\end{lemma}

\subsection{Discretization of the functional $J$} 
{We recall that for a convex function $v$ on $\Omega$, the one-sided 
partial derivatives are defined everywhere and denote by $\nabla_- 
v$ the left-hand side gradient of $v$. Since the convex function $v$ is 
differentiable almost everywhere}, we have
\begin{equation} \label{new-conv}
J(w) = \int_{\Omega} \Phi(\nabla_- w) \ud x.
\end{equation}
We now use the latter expression of $J(w)$ to construct a discrete analogue of 
$J(w)$ using Riemann sums. 

For $x = (x_1, \ldots, x_n) \in \mathcal{M}_h$, we define
$$
P_x = \{ \, y \in \bar{\Omega}, x_i-h \leq y_i \leq x_i \, \}.
$$
{It is easy to see that $P_x$ is nonempty if and only if either 
$\dist(x,\partial\Omega) < h\sqrt{n}$ or $\dist(x-h,\partial\Omega)< 
h\sqrt{n}$. Furthermore, we have} $\cup_{x \in \mathcal{M}_h} P_x = 
\bar{\Omega}$ and for $x \neq y, P_x \cap P_y$ is a set of Lebesgue measure 0. 
For $x \in  \mathcal{M}_h$, we denote by $|P_x|$ the Lebesgue measure of $P_x$, 
i.e. $|P_x| = h^n$ if $P_x \cap \Omega \neq \emptyset$. We use the notation $x-h/2$ to 
denote the center of $P_x$ {that is, the point obtained by 
subtracting $h/2$ from each coordinate $x_i$ of $x$}.

We define for $v_h \in \mathcal{U}_h$ and a convex real-valued function $\Phi$ 
on $\R^n$
$$
J_h(v_h) =  \sum_{x \in \mathcal{M}_h^{}} |P_x| \Phi (\nabla_h v_{h}(x)), 
$$
where $\nabla_h v_{h}$ is the vector of backward finite differences of the mesh  
function $v_h$, i.e.
$$
\nabla_h v_{h}(x) = (\partial^i_- v_h(x))_{i=1,\ldots,n}.
$$ 
We {assume} that the sum in the definition of 
$J_h(v_h)$ is over the set of mesh points at which $\nabla_h v_{h}(x)$ is 
defined.
We note that for all $i$, $\partial^i_- v_h$ extends to a piecewise constant  
function on $\Omega$, denoted also by $\partial^i_- v_h$ and taking the constant 
value $\partial^i_- v_h(x)$ on $P_x$ for $x \in  \mathcal{M}_h$.

\begin{lemma} \label{lem-conv}
Let $v \in C^{2}(\tir{\Omega})$ and $v_h$ a family of mesh functions. We have
\begin{align} \label{stab-J}
\max_{x \in \mathcal{M}_h^{}} \bigg|\nabla_h v_{h}(x)  - \nabla_- v(x) \bigg|  
\to 0 \, \text{as} \, h \to 0 \, \text{implies} \, J_h(v_h) \to J(v) \, 
\text{as} \, h \to 0.
\end{align}
\end{lemma}

\begin{proof}
For $v \in C^{2}(\tir{\Omega})$, $\nabla v(x)$ is defined, is equal to 
$\nabla_- v(x)$ and is uniformly bounded on $\tir{\Omega}$. We have by 
definition of the integral
\begin{align}
J(v) & = \lim_{h \to 0} \sum_{x \in \mathcal{M}_h^{}}  |P_x| \Phi
\bigg(\nabla_- v \big(x - \frac{h}{2} \big) \bigg) \label{J-limit}.
\end{align}
Moreover
\begin{align*}
 \sum_{x \in \mathcal{M}_h}  |P_x|  \Phi\bigg(\nabla_- v \big(x - \frac{h}{2} 
\big) \bigg) - J_h(v_h) & = h^n \sum_{x \in \mathcal{M}_h^{}} 
\bigg(\Phi\bigg(\nabla_- v \big(x - \frac{h}{2} \big) \bigg) - \Phi(\nabla_h 
v_h(x)) \bigg). 
\end{align*}
And so by the $C^1$ continuity of $\Phi$ and the mean value theorem
\begin{align*}
\bigg|h^n \sum_{x \in \mathcal{M}_h} \Phi\bigg(\nabla_- v \big(x - \frac{h}{2} 
\big) \bigg) - J_h(v_h) \bigg| & \leq C \max_{x \in \mathcal{M}_h} 
\bigg|\Phi\bigg(\nabla_- v \big(x - \frac{h}{2} \big) \bigg) - \Phi(\nabla_h v_h 
(x))\bigg| \\
& \leq C \max_{x \in \mathcal{M}_h} \bigg|\nabla_- v \big(x - \frac{h}{2} ) - 
\nabla_h v_h (x)\bigg|.
\end{align*}
On the other hand
$$
\bigg|\nabla_- v \big(x - \frac{h}{2} ) -  \nabla_h v_h  (x)\bigg| \leq 
\bigg|\nabla_- v \big(x - \frac{h}{2} ) - \nabla_- v(x) \bigg| + 
\bigg|\nabla_- v(x) - \nabla_h v_h (x)\bigg|.
$$

By assumption $\max_{x \in \mathcal{M}_h} |\nabla_h v_{h}(x) - \nabla_- v(x) | 
\to 0 \, \text{as} \, h \to 0$. By the $C^1$ continuity of $v$, $\bigg|\nabla_- 
v \big(x - \frac{h}{2} ) - \nabla_- v(x) \bigg| \to 0$ as $h \to 0$.
Thus using \eqref {new-conv} and \eqref{J-limit}, we obtain $J_h(v_h) \to J(v) 
\text{ as } h \to 0$.
\end{proof}

\begin{rem}
If $v\in C^2(\tir{\Omega})$, we have for all $i=1,\ldots,n$ and  $x \in 
\mathcal{M}_h^{}$ by a Taylor series expansion
$$
\bigg|\frac{\partial v}{\partial x_i}(x) - \partial_-^i v (x) \bigg| \leq C h,
$$
where $C$ depends only on the maximum of $\partial^2 v/(\partial x_i^2)$ on 
$\tir{\Omega}$. Moreover
$$
\bigg|\frac{\partial v}{\partial x_i}(x-\frac{h}{2}) - \partial_-^i v (x) 
\bigg| \leq \bigg| \partial_-^i v (x) -\frac{\partial v}{\partial x_i}(x)\bigg| 
+ \bigg|\frac{\partial v}{\partial x_i}(x) - \frac{\partial v}{\partial 
x_i}(x-\frac{h}{2}) \bigg|.
$$
Using the $C^1$ continuity of \edit{$\partial v/\partial x_i$}, we obtain that 
the condition 
$\max_{x \in \mathcal{M}_h^{}} |\nabla_h v_{}(x) - \nabla_- v(x-h/2)| \to 0$ 
holds for $v\in C^2(\tir{\Omega})$ \edit{ and $v_h$ the mesh function obtained 
by restriction of $v$ onto the set $\mathcal{M}_h$.}
\end{rem}

\subsection{Standard discretization for $\det \nabla^2 u$ and local discrete 
convexity}

We seek to minimize $J_h$ over \edit{the following discrete counterpart of the 
set $S$ defined in \eqref{s}}
% \begin{align} 
\begin{multline}\label{s-h}
S_h = \{\, v_h \in \mathcal{U}_h, v_h= g \, \text{ on } \,  \partial 
\mathcal{M}_h,~ H[v_h](x) \geq 0  
\text{ and } \,
(\det H[v_h](x))^{\frac{1}{n}} \geq f(x)^{\frac{1}{n}} \text{ for all } x \in 
\mathcal{M}_h^{\circ}  \, \}.
\end{multline}
% \end{align}
This amounts to minimizing a strictly convex functional over the convex set 
$S_h$. Thus the main difficulty we face is to show that the set $S_h$ is 
nonempty. Our approach for proving that the set $S_h$ is nonempty is { to prove 
the existence of a solution of the finite difference system} 
\begin{align}
\begin{split} \label{Dirich-prob}
\det H[\hat{u}_h](x) & = f(x), x \in \mathcal{M}_h^{\circ} \\
H[\hat{u}_h](x) & \geq 0, x \in \mathcal{M}_h^{\circ} \\
\hat{u}_h & = g \, \text{on} \,  \partial \mathcal{M}_h.
\end{split}
\end{align}
The idea has been used in \cite{Awanou-Iso}. Let us assume that $0 < c_0 \leq f 
\leq c_1$ for constants $c_0$ and $c_1$. It is shown in \cite[Theorem 
3.4]{Awanou-k-Hessian12} that if $u \in C^4(\tir{\Omega})$ is a solution of 
\eqref{m1}, then the finite difference system \eqref{Dirich-prob} has a 
unique solution $\hat{u}_h$ in the ball
$$
B_{\rho}( u) = \{ \, v_h \in \mathcal{U}_h, ||v_h- u ||_{2,\infty} \leq \rho\, \},
$$
for 
$$
\rho = C ||u||_{C^{4}(\Omega)} h^2,
$$
and a constant $C$. 

\begin{prop} \label{lem-smooth}
Suppose that  the convex solution $u$ of \eqref{m1} is in $C^4(\tir{\Omega})$ 
and strictly convex. Then the functional $J_h$ has a unique minimizer $u_h$ in 
$S_h$ and 
%where $\hat{u}_h$ solves \eqref{Dirich-prob}. Thus 
$\inf_h J_h(u_h ) \leq J(u)$.
\end{prop}

\begin{proof}
It is shown in \cite[Theorem 3.4]{Awanou-k-Hessian12} that there exists 
$\hat{u}_h \in \mathcal{U}_h$ which solves \eqref{Dirich-prob}.
It follows that $\hat{u}_h \in S_h$ and thus $S_h$ is nonempty. By the strict 
convexity of $J_h$ we conclude that $J_h$ has a unique minimizer $u_h$ in $S_h$.
Since $|| \hat{u}_h -u||_{2,\infty} \leq C h^2$, 
we have
$\max_{x \in \mathcal{M}_h^{\circ}} |\nabla_h \hat{u}_h (x) - \nabla u(x)| \to 
0$  as $h \to 0$. \edit{Recall that $\hat{u}_h (x) = u(x)$ on $\partial 
\mathcal{M}_h$; therefore for $x \in \partial \mathcal{M}_h$ with $\nabla_h 
\hat{u}_h (x)$ defined, we have
\begin{align*}
|\nabla_h \hat{u}_h (x)-\nabla u(x)| & \leq C \max_{j}\big|\frac{\hat{u}_h(x)
- \hat{u}_h(x-h e_j)}{h} - \frac{\partial u}{\partial x_j}(x) \big| \\
  & = C \max_{j} \big|\frac{u(x) - \hat{u}_h(x-h e_j)}{h} - 
\frac{\partial u}{\partial x_j}(x) \big| \\
  & \leq C \max_{j} \big|\frac{u(x) - u(x-h e_j)}{h} - 
\frac{\partial u}{\partial x_j}(x) \big| + C \max_{j} \big|\frac{u(x-h e_j) 
- \hat{u}_h(x-h e_j)}{h}  \big|\\
& \le C 
\max_j\sup_{x\in\bar{\Omega}}\Big|\frac{\partial^2u}{\partial^2x_j}\Big| h + 
C \big|\hat{u}_h-u\big|_{0,\infty}.
\end{align*}
Since $|\hat{u}_h-u|_{0,\infty} \leq C h^2$ and $u$ is $C^2$ with uniformly  
bounded second derivatives, we conclude that 
$$
\max_{x\in\mathcal{M}_h} |\nabla_h \hat{u}_h(x)-\nabla u(x)| \to 0,
\text{ as }  h \to 0.
$$
}
Thus by Lemma \ref{lem-conv} we have $J_h(\hat{u}_h) \to J(u)$. 

We now prove, using a contradiction argument, that
\begin{equation*} %\label{inf}
\inf_h J_h( \hat{u}_h ) = J(u).
\end{equation*}
If $J(u) < \inf_h J_h(\hat{u}_h)$, then there is number $a$ such that  
$J(u) < a < \inf_h J_h(\hat{u}_h) \leq J_h(\hat{u}_h)$ for all $h$. Thus 
$\lim_{h \to 0} J_h(\hat{u}_h) \geq a > J(u)$ a contradiction.

On the other hand, if $\inf_h J_h(\hat{u}_h) < J(u)$, then there is number $b$ 
such that $\inf_h J_h(\hat{u}_h) < b < J(u)$. By definition of the infimum, we 
can find a subsequence $h_k$ such that 
$\inf_h J_h(\hat{u}_h) \le J_{h_k}(\hat{u}_{h_k})\leq b  < J(u)$. This also 
leads to a contradiction. Finally, since $u_h$ is a minimizer of $J_h$  it 
follows that $J_h(u_h) \leq J_h( \hat{u}_h )$  and $\inf_h J_h(u_h ) \leq J(u)$.
\end{proof}

To $v_h \in S_h$, we associate a Borel measure $M[v_h]$ defined by
$$
M[v_h](B) = h^n \sum_{x \in B \cap \mathcal{M}_h} \det H[v_h](x).
$$
The following lemma is proved in \cite{AwanouAwi2}.
\begin{lemma} \label{awi}
Let $v_h \in S_h$ converge uniformly on compact subsets to $v \in C(\tir{\Omega})$ and convex. Then $M[v_h]$ converges weakly to $\det \nabla^2 v$.
\end{lemma}

The following lemma is the only place in the paper where we use the result of Lions which gives a variational approach to the Aleksandrov solution of \eqref{m1}.
Arguing as in \cite{Aguilera2008}, we have

\begin{lemma} \label{Aguil-type}
Under the assumptions of Proposition \ref{lem-smooth}, we have
$$
\inf_h J_h(u_h ) = J(u).
$$
\end{lemma}

\begin{proof}
It remains to prove that $J(u) \leq \inf_h J_h(u_h )$. The technique to prove such a result was given in \cite[Section 5]{Aguilera2008}. %It is reproduced here for the convenience of the reader.
We make an essential use of Lemma \ref{awi} and the result of Lions \cite{Lions86}.

Let us define
\begin{align*}
|v_h|_{0,\infty}' & = \max_{x \in \mathcal{M}_h^{}} |v_h(x)|, \\
|v_h|_{1,\infty}' & = \max \{ \, |\partial^i_{+} v_h|_{0,\infty}', i=1,\ldots,n \, 
\}, \\
 |v_h|_{2,\infty}' & = \max \{ \,  |\partial^i_{+} \partial^i_{-}  
v_h|_{1,\infty}',  |\partial^i \partial^j v_h|_{1,\infty}',  i,j=1,\ldots,n \,\}, 
\\
||v_h||_{2,\infty}' & =  \max \{ \,   |v_h|_{0,\infty}',   |v_h|_{1,\infty}',   
|v_h|_{2,\infty}' \, \}.
\end{align*}
We make the assumption that the terms appearing in the definition of the above norms and semi norms are those for which the indicated discrete derivatives are defined.

For $K >0$, put
\begin{align*}
%\Lambda^0_{h,K} & = \{ \, v_h \in \mathcal{U}_h, |v_h|_{0,\infty}'  \leq K \, \} \\
%\Lambda^1_{h,K} & = \{ \, v_h \in \mathcal{U}_h, |v_h|_{0,\infty}'  \leq K \text{ and }  |v_h|_{1,\infty}'  \leq K  \, \} \\
\Lambda^2_{h,K} & = \{ \, v_h \in \mathcal{U}_h, ||v_h||_{2,\infty}'  \leq K \, \}.
\end{align*}
Moreover, put $S_{h,K} = S_h \cap \Lambda^2_{h,K}$. 
%We denote by $\Lambda^2_K$ the set of functions with weak derivatives up to order 2 bounded by $K$.
%{Dirich-prob}

Since $u\in C^4(\bar{\Omega})$, there exist $K>0$ and $h_0>0$ such 
that  $\hat{u}_h \in \mathcal{S}_h\cap\Lambda^2_{h,K}$, where $\hat{u}_h$ is 
the solution of \eqref{Dirich-prob}. This proves that there exists $K>0$ and $h_0 > 0$ such that $S_{h,K} $ is nonempty for all $0<h<h_0$.

Let $u_h^K$ denote the unique minimizer of $J_h$ in $S_{h,K}$. By \cite[Theorem 4.5]{Aguilera2008}, there exists a subsequence $u_{h'}^K$ and a $C^2$ convex function $u^K$
such that $ |u_{h'}^K -u^K|_{1,\infty}' \to 0$ and $ |u_{h'}^K -u^K|_{0,\infty}' \to 0$ as $h' \to 0$. 

We prove that 
$u^K=g$ on $\partial \Omega$. For $x\in \partial\Omega$, there is a family $x_{h'} \in \partial
\mathcal{M}_{h'}$ such that $x_{h'}$ converges to $x$. Thus
\begin{align*}
|u^K(x)-g(x)|& \le|u^K(x)-u^K(x_{h'})|+|u^K(x_{h'})-u_{h'}(x_{h'})|+|u_{h'}(x_h)-g(x)| \\
            & \le |u^K(x)-u^K(x_{h'})|+|u^K(x_{h'})-u_{h'}(x_{h'})|+|g(x_{h'})-g(x)| \\
           &  \le |u^K(x)-u^K(x_{h'})| + |u_{h'}^K -u^K|_{0,\infty}'  + |g(x_{h'})-g(x)|
\end{align*}
Passing to the limit as $h' \to 0$ yields $u^K(x)=g(x)$ by continuity of
$g$ and $u^K$ and the fact that $|u_{h'}^K -u^K|_{0,\infty}' $ as $h' \to 0$.  

By Lemma 
\ref{lem-conv} $J_{h'}(u_{h'}^K) \to J(u^K)$ and by Lemma \ref{awi} we have
$\det D^2 u^K \geq f$ in the sense of measures. Thus $u^K \in S$.

Since $u^K \in S$ and $J(u) = \inf_{v \in S} J(v)$, we have
$$
J(u) \leq J(u^K)= \inf_{h'} J_{h'} (u_{h'}^K).
$$ 
We may assume that $ \inf_{h'} J_{h'} (u_{h'}^K) = \inf_h J_h(u_h^K)$. This is because, using the definition of infimum, the sequence $u_{h'}^K$ can be chosen to satisfy that property before passing to a subsequence converging to $u^K$.
We conclude that
$$
J(u) \leq \inf_{h, K} J_h(u_h^K). % = \inf_h J_h(u_h).
$$
For a fixed $h$, we can find $K$ such
that $u_h$ is in $S_{h,K}$. Here $K$ depends on $h$. To see this, note that since
$h$ is fixed, the number of mesh points is finite and $u_h$ takes real
values. Thus $K$ can simply be taken as any umber greater than
$||u_h||'_{2,\infty}$. We conclude that $ \inf_{h, K}  J_h(u_h^K) = \inf_h J_h(u_h)$. This implies that
$$
J(u) \leq \inf_h J_h(u_h),
$$
and concludes the proof.
\end{proof}

For $\Phi(p) = |p|^2$ we can give a more precise result.

\begin{thm}  \label{thm-smooth}
Suppose that  the convex solution $u$ of \eqref{m1} is in $C^4(\tir{\Omega})$ 
and strictly convex. Then for $\Phi(p) = |p|^2$, the unique minimizer $u_h$ in 
$S_h$ of the functional $J_h$ satisfies
$$
||u_h-u||_{0} \to 0 \text{ as } h \to 0.
$$ 
\end{thm}

\begin{proof}

Let $\mathcal{M}_h^-$ denote the subset of $\mathcal{M}_h$ of mesh points $x$ at which $\nabla_h v_h(x)$ is defined for $v_h \in \mathcal{U}_h$.

We first establish that for $v_h \in \mathcal{U}_h$, we have
\begin{equation} \label{var-ineq}
\sum_{x \in \mathcal{M}_h^- } \< \nabla_h u_h(x), \nabla_h v_h(x) \> \geq \sum_{x \in \mathcal{M}_h^- }  |\nabla_h u_h(x)|^2.
\end{equation}
For $t \in [0,1]$, define
$$
\phi(t) = J_h(u_h + t(v_h-u_h)) = \frac{h^n}{2} \sum_{x \in \mathcal{M}_h^- }   |\nabla_h u_h(x) + t \nabla_h (v_h-u_h) (x) |^2.
$$
Clearly $\phi(t) \geq 0$ and $\phi$ is continuous. It is not difficult to check that for two vectors $r$ and $s$,
$$
\frac{d}{d t} |r+t s|^2 = 2 t |s|^2 + 2 \<r,s\>.
$$
This implies that $\phi$ is $C^1$. Since $\phi(t)$ achieves its minimum at $t=0$, we have
$$
0 \leq \phi'(0) = h^n \sum_{x \in \mathcal{M}_h^- } \<\nabla_h u_h(x), \nabla_h (v_h-u_h) (x) \>,
$$
from which \eqref{var-ineq} follows. 

%For $x \in \mathcal{M}_h^-$, 
By \eqref{var-ineq} we have
\edit{
\begin{align*}
|u_h-u|_1^2&=h^n \sum_{x \in \mathcal{M}_h^- } |\nabla_h (u_h-u) (x) |^2  \\
& = h^n \sum_{x \in \mathcal{M}_h^- }  |\nabla_h u_h(x)|^2 +  |\nabla_h u(x)|^2 -2 \<\nabla_h u_h(x), \nabla_h u(x) \> \\
& \leq h^n \sum_{x \in \mathcal{M}_h^- }  -|\nabla_h u_h(x)|^2 +  |\nabla_h u(x)|^2 \\
& \le J_h(u) - J_h(u_h)\\
&\le  \big|J_h(u)-J(u)\big| + \big|J(u)-J_h(u_h)\big|
\end{align*}
}
It then follows from \edit{Lemma \ref{lem-smooth} and }Lemma \ref{Aguil-type} that $|u_h-u|_1 \to 0$ as $h \to 0$. Since $u_h-u=0$ on $\partial \mathcal{M}_h$ the result follows from Poincare's inequality, Lemma \ref{poincare}.

\end{proof}

\subsection{Monotone discretization of $\det \nabla^2 u$ and wide stencil 
convexity}

We prove in this section that if one uses the monotone scheme introduced in 
\cite{Oberman2010a}, the minimizer of the discrete optimization problem is a 
solution of the corresponding finite difference equations.  Following 
\cite{Oberman2010a}, we define a monotone Monge-Amp\`ere operator by 
\begin{equation} \label{disc-det}
M[v_h](x) = \inf_{(\alpha_1,\ldots,\alpha_n) \in W_h(x)} \prod_{i=1}^n 
\frac{v_h(x+ \alpha_i) -2 v_h(x) + v_h(x-\alpha_i)}{|\alpha_i|^2}\enskip  
\text{ for } x \in \mathcal{M}_h^{\circ},
\end{equation}
where $W_h(x)$ denotes the set of orthogonal bases of $\R^n$ such that for 
$(\alpha_1,\ldots,\alpha_n) \in W_h(x)$, $x\pm \alpha_i \in 
\mathcal{M}_h^{\circ}, \forall i$. 

\begin{defn}
We say that a mesh function $v_h$ is {\it wide stencil convex} if and only if  
$\Delta_e v_h(x) = v_h(x+e) -2 v_h(x) + v_h(x-e) \geq 0$ for all $x \in 
\Omega_h$ and  $e \in \mathbb{Z}^n$ for which $\Delta_e v_h(x)$ is defined. 
\end{defn}
We recall that the discrete Laplacian takes the form
$$
\Delta_h v_h (x) = \sum_{i=1}^d \Delta_{h e_i} v_h(x),
$$
where $\{ \, e_i, i=1,\ldots,d \, \}$ denotes the canonical basis of $\R^d$.

Let $\mathcal{C}_h$ denote the cone of wide stencil convex mesh functions. 
In this section, we will refer to a wide stencil convex function simply as a discrete convex function. We also make the slight abuse of notation of denoting by $S_h$ the discrete version of the set $S$ using the notion of
wide stencil convexity and the monotone discretization of $\det \nabla^2 u$, i.e.
\begin{align} \label{s-h-2}
\begin{split}
S_h & = \{\, v_h \in \mathcal{C}_h, v_h= g \, \text{on} \,  \partial \mathcal{M}_h, 
\, \text{and} \,
(M[v_h](x) )^{\frac{1}{n}} \geq f(x)^{\frac{1}{n}}, x \in \mathcal{M}_h^{\circ} \, \}.
\end{split}
\end{align}
We seek a minimizer of $J_h$ over $S_h$. As in the previous section we consider the problem: find $u_h \in \mathcal{C}_h$ such that
\begin{align} \label{mono-h}
\begin{split}
M[u_h](x) & = f(x), x \in \mathcal{M}_h^{\circ} \\
u_h & = g \, \text{on} \, \partial \mathcal{M}_h.
\end{split}
\end{align}
When $f \in C(\Omega)$, it is shown in \cite{Oberman2010a} that $u_h$ converges 
uniformly on compact subsets to the so-called viscosity solution of \eqref{m1}  
when \eqref{m1} has a unique one. One can add a perturbation $\epsilon u_h$ to 
the operator to force uniqueness. But \eqref{mono-h} may have more than one 
solution. Here we prove that  %the minimizing procedure selects a unique 
there is no uniqueness issue with the variational framework proposed in this paper.

\begin{thm} \label{first-good}
For $\Phi(p)=|p|^2$, the functional $J_h$ has a unique minimizer $u_h$ in $S_h$  
and $u_h$ solves the finite difference equation \eqref{mono-h}.
\end{thm}

\begin{proof}
We first prove that the set $S_h$ is convex. We note that for $\lambda \geq 0$
$$
(M[\lambda v_h](x))^{\frac{1}{n}} = \lambda (M[ v_h](x))^{\frac{1}{n}}.
$$
It is therefore enough to prove that for $v_h, w_h \in \mathcal{C}_h$, we have 
\begin{equation} \label{conc}
 (M[ v_h + w_h ](x))^{\frac{1}{n}} \geq  (M[ v_h](x))^{\frac{1}{n}} +  
 (M[ w_h](x))^{\frac{1}{n}}. 
\end{equation}
Let $(\alpha_1,\ldots,\alpha_n) \in W_h(x)$ be such that 
\begin{align*}
M[ v_h + w_h ](x) & = \prod_{i=1}^n (\lambda_1^i + \lambda_2^i) 
\intertext{ with }
 \lambda_1^i  = \frac{v_h(x+ \alpha_i) -2 v_h(x) + 
v_h(x-\alpha_i)}{|\alpha_i|^2} &\text{ and }
 \lambda_2^i  = \frac{w_h(x+ \alpha_i) -2 w_h(x) + 
w_h(x-\alpha_i)}{|\alpha_i|^2}.  
\end{align*}
Since $v_h, w_h \in \mathcal{C}_h$, $\lambda_1^i,  \lambda_2^i \geq 0$ for all 
$i$.  By Minkowski's determinant theorem,
$$
\bigg( \prod_{i=1}^n (\lambda_1^i + \lambda_2^i) \bigg)^{\frac{1}{n}} \geq 
\bigg( \prod_{i=1}^n \lambda_1^i \bigg)^{\frac{1}{n}}  + \bigg( \prod_{i=1}^n  
\lambda_2^i \bigg)^{\frac{1}{n}}, 
$$
from which \eqref{conc} follows. 

We recall that Problem \eqref{mono-h} was shown in \cite{Oberman2010a} to have a solution. Thus the set $S_h$ is nonempty.
\edit{Since $\Phi$ is strictly 
convex by assumption, it follows that the functional $J_h$ has a 
unique minimizer $u_h$ on the convex set $S_h$.

We now show that $u_h$ solves the finite difference system \eqref{mono-h}. To 
this end, it suffices to show that 
$$M[u_h]=f \text{ on } \mathcal{M}_h^\circ.$$
}
Let us assume \edit{ to the contrary } that there exists $x_0 \in 
\mathcal{M}_h^{\circ}$ such that
$$
M[u_h](x_0) > f(x_0) \geq 0.
$$
This implies that for all $e \in \mathbb{Z}^d$, $ \Delta_e u_h(x_0) >0$. 
Let $\epsilon_0 = \inf \{ \,  \Delta_e u_h(x_0), e \in \mathbb{Z}^d \}$
and $\epsilon_1 =M[u_h](x_0) -  f(x_0)$. Finally, put $\epsilon=\min( \epsilon_0
, \epsilon_1)$. We define $w_h$ by
$$
w_h(x) = u_h(x), x \neq x_0, w_h(x_0) = u_h(x_0) + \frac{\epsilon}{4}.
$$
By construction $w_h= g \, \text{on} \,  \partial \mathcal{M}_h$. For 
$x \neq x_0$, $\Delta_e w_h(x)= \Delta_e u_h(x)$ or $\Delta_e w_h(x)= \Delta_e 
u_h(x) + \epsilon/4$. Moroever 
$\Delta_e w_h(x_0)= \Delta_e u_h(x_0)- \epsilon/2 \geq \epsilon_0- \epsilon/2 
\geq \epsilon/2 >0$ by the definition of $\epsilon$. We conclude that $w_h \in 
\mathcal{C}_h$.

For $x \neq x_0$, $M[w_h](x) \geq M[u_h](x)$ and $M[w_h](x_0) = M[u_h](x_0)  
-\epsilon/2$. Thus $M[w_h](x) \geq f(x)$ for all $x \in  \mathcal{M}_h^{\circ}$.
It remains to show that $J_h(w_h) < J_h(u_h)$. Let $\mathcal{M}_{x_0}$ denotes 
the subset of  $\mathcal{M}_h$ consisting in $x_0$ and the points $x_0 + h e_j, 
j=1,\ldots,d$ at which $\nabla_h u_h$ is defined.
We have
\begin{align*}
J_h(w_h) & = h^d \sum_{x \notin \mathcal{M}_{x_0}} |\nabla_h u_h(x)|^2 + h^d  
|\nabla_h u_h(x_0)|^2 +  \sum_{j=1}^d  |\nabla_h u_h(x_0+ h e_j)|^2 \\
J_h(w_h) & = h^d \sum_{x \notin \mathcal{M}_{x_0}} |\nabla_h u_h(x)|^2 + h^{d-2}
\sum_{i=1}^d (w_h(x_0) - w_h(x_0-h e_i))^2 \\
& \quad + h^{d-2} \sum_{j=1}^d \sum_{i=1}^d (w_h(x_0+h e_j) - 
w_h(x_0 +h e_j - h e_i))^2 \\
& = h^d \sum_{x \notin \mathcal{M}_{x_0}} |\nabla_h u_h(x)|^2  + h^{d-2} 
\sum_{j=1}^d \sum_{i=1 \atop i \neq j}^d (w_h(x_0+h e_j) - w_h(x_0+h e_j - h e_i 
))^2 \\
& \quad +  h^{d-2} \sum_{i=1}^d (w_h(x_0) - w_h(x_0-h e_i))^2 + 
(w_h(x_0+h e_i) - w_h(x_0 ))^2.
\end{align*}
However
\begin{multline*}
\sum_{i=1}^d (w_h(x_0) - w_h(x_0-h e_i))^2 + (w_h(x_0+h e_i) - w_h(x_0))^2 = \\
\sum_{i=1}^d (u_h(x_0) - u_h(x_0-h e_i))^2 + (u_h(x_0+h e_i) - u_h(x_0))^2 + 
\frac{d \epsilon^2}{8} - \frac{\epsilon}{2} \Delta_h u_h(x_0). 
\end{multline*}
Thus, \edit{by our choice of $\epsilon$}
\begin{align*}
J_h(w_h) & = J_h(u_h) + \frac{d \epsilon^2}{8} -   \frac{\epsilon}{2} 
\Delta_h u_h(x_0) = J_h(u_h) + \frac{\epsilon}{2}(  \frac{d \epsilon}{4} - 
\Delta_h u_h(x_0)) \\
& <  J_h(u_h),
\end{align*}
\st{by the choice of $\epsilon$.} Indeed, since $ \Delta_e u_h(x_0) \geq 
\epsilon_0$, we get $\Delta_h u_h(x_0)) \geq d \epsilon_0 \geq d \epsilon > d 
\epsilon /4$ \edit{which} \st{. This} contradicts the assumption that $u_h$ is 
a minimizer and concludes the proof.
\end{proof}

\section{Numerical results} \label{sec:num}

In this section, we report numerical results for the variational framework 
proposed in the paper for the 2D problem. For most of the numerical results we  
use a standard discretization  for $\det \nabla^2 u$ and local discrete 
convexity. We also include numerical results for a more general situation where 
the right hand side of \eqref{m1} is a measure. Previously published results on 
the monotone scheme are not satisfactory for this case \cite{Benamou2014b}.

Throughout this section $\bar{\Omega}$ is the unit square $[0,1]\times [0,1]$. 
We recall that the optimization problem of interest to us, in the case of the 
standard finite difference discretization, is the following:
\begin{equation}\label{eq:opt}
\begin{array}[m]{rll}
 \text{Minimize ~~~~~}& 
h^2 \sum\limits_{x \in \mathcal{M}_h } \Phi(\nabla_h u_h(x)) \\
 \text{ subject to }& ~~ u_{h} = g_{} & \text{ on } \partial 
\mathcal{M}_h \\[2mm]
 &\lambda_{min}(H [u_{h}](x)) \ge 0 & \text{ 
for }x\in\mathcal{M}_h^{\circ}\\[2mm]
 &~~ \sqrt{\det(H [u_{h}](x))}\ge \sqrt{f_{}(x)} & \text{ for }  x \in 
 \mathcal{M}_h^{\circ},
 \end{array}
\end{equation}
where $u_h$ is the unknown variable, $\lambda_{min}$ is the smallest eigenvalue 
of the matrix $H [u_{h}](x)$ and we recall that $H [u_{h}](x)$ denotes the 
discrete Hessian. 

\subsection{Solvability and implementation} 

Under the convexity assumption on $\Phi$,  since the operators 
$\sqrt{\det(\cdot)}$ and $\lambda_{min}(\cdot)$ are concave, it is easily 
verified that \eqref{eq:opt} is a convex optimization program. Therefore, we are 
guaranteed that any algorithm that finds a local minimizer recovers a global 
minimizer. Furthermore, the global minimizer will be unique if we choose $\Phi$ 
to be strictly convex. A variety of methods and algorithms to solve convex 
programs like \eqref{eq:opt}, including primal-dual and barrier methods, are 
readily available in the literature. It is not our goal in this section to 
develop or identify the most efficient method for solving \eqref{eq:opt}.
Instead, we aim to provide numerical evidence supporting the analysis done in 
section \ref{sec:disc}. For rapid prototyping, we use MATLAB and take advantage 
of the fact that our convex program is a typical ``disciplined convex program'' 
as introduced in the  convex optimization toolbox CVX \cite{Grant, cvx}. In CVX, 
the user has the choice between several solvers and we choose SDPT3  
\cite{SDPT3}. It implements an infeasible primal-dual path-following algorithm.

For computational expedience but at a cost of increased problem size, CVX 
internally converts problem \eqref{eq:opt} -- see for example \cite{Grant} and 
the references therein -- to the canonical form  
\begin{equation}\label{eq:optc}
\begin{split}
  \text{Minimize }& c^Tx+d\\
  \text{subject to } & \boldsymbol{A}x=b\\
  & x\in \mathcal{S},
\end{split}  
\end{equation}
where $\mathcal{S}$ is a convex set, $x\in\R^m$ is the unknown, and the 
parameters $\boldsymbol{A}\in \R^{k\times m}$, $b\in \R^k$, $c\in \R^m$ and 
$d\in\R$ are computed from the original problem. The canonical problem \eqref{eq:optc} is then solved using the solver SDPT3.

\subsection{Test cases and results} 
We provide numerical evidence for four scenarios. The data of the corresponding 
Monge-Amp\`ere  Dirichlet problems are  given in \textsc{Table} \ref{tbl:data}. 
The results are reported in \textsc{Tables} \ref{tbl:1}, \ref{tbl:2}  and  
\textsc{Figure} \ref{fig:fig}.
\begin{table}[!h]
\begin{tabular}{ccll}
\hline\\[-1mm]
  Test 1 & \enskip & $f(x,y)=(1+x^2+y^2)\mathrm{e}^{x^2+y^2}$ & 
$g(x,y)=\mathrm{e}^{(x^2+y^2)/2}$\\[2mm]
  Test 2 & & $f(x,y)=\dfrac{2}{(2-x^2-y^2)^2}$ & 
$g(x,y)=-\sqrt{2-x^2-y^2}$\\[2mm]
  Test 3 & & $f(x,y)=4$ &  $g(x,y)=(x-1/2)^2+(y-1/2)^2$ \\[2mm]
  Test 4 & & $ f(x,y)=0$ &  $g(x,y)=|x-1/2|$\\[2mm] \hline
\end{tabular}

\caption{Data for the numerical experiments with $\Omega = (0,\,1)\times 
(0\,1)$.  \label{tbl:data} }
\end{table}

\begin{figure} [!ht]
\subfloat[Numerical solution for Test 
1.]{\includegraphics[width=.42\linewidth]{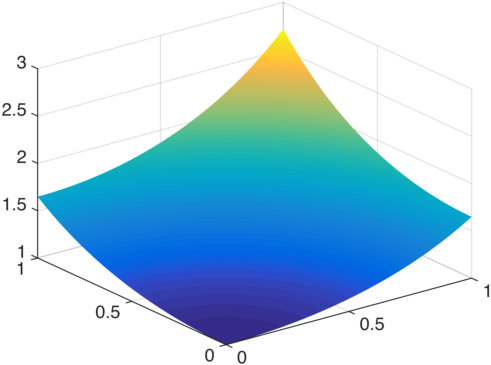}}\quad
\subfloat[Numerical solution for Test 
2.]{\includegraphics[width=.42\linewidth]{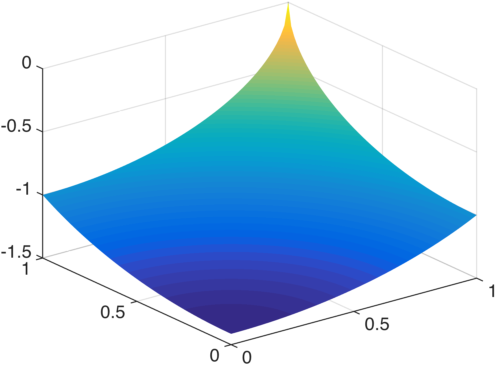}}\\[3mm]
\subfloat[Numerical solution for Test 
3.]{\includegraphics[width=.42\linewidth]{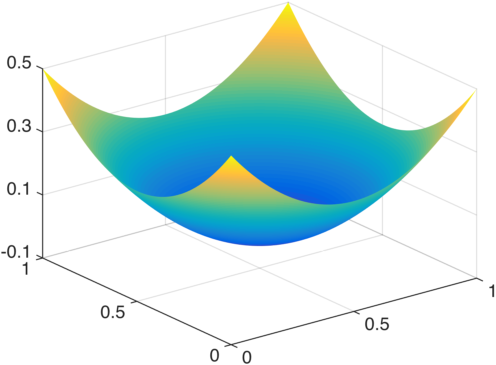}}\quad
\subfloat[Numerical solution for Test 
4.]{\includegraphics[width=.42\linewidth]{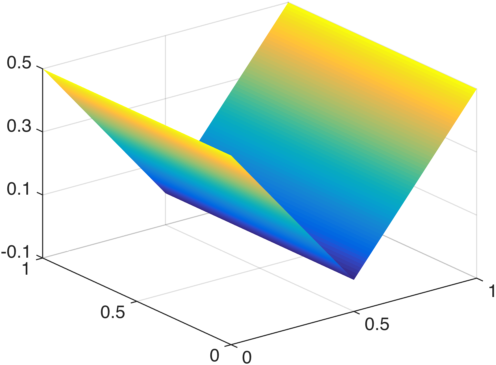}}
\caption{Solutions computed on a $65\times 65$ grid with $\Phi(x,y)=\sqrt{x^2+y^2}$ using the disciplined
convex programming toolbox CVX with the solver SDPT3.  \label{fig:fig}}
\end{figure}

\begin{table}[!h] 
\centering\setstretch{1.2}
 \begin{tabular}{ccccc }
\hline
$h$ & Test 1 & Test 2 & Test 3 & Test 4 \\ \hline
$2^{-2}$&3.9093$\times 10^{-3}$ & 2.5104$\times 10^{-2}$ & 2.9143$\times 10^{-16}$ & 1.7233$\times 10^{-5}$ \\ \hline
$2^{-3}$&1.0340$\times 10^{-3}$ & 2.6475$\times 10^{-2}$ & 1.1102$\times 10^{-16}$ & 3.8580$\times 10^{-15}$ \\ \hline
$2^{-4}$&2.6643$\times 10^{-4}$ & 2.2113$\times 10^{-2}$ & 5.5511$\times 10^{-17}$ & 1.0963$\times 10^{-14}$ \\ \hline
$2^{-5}$&6.6964$\times 10^{-5}$ & 1.6920$\times 10^{-2}$ & 3.0531$\times 10^{-16}$ & 1.5155$\times 10^{-14}$ \\ \hline
$2^{-6}$&1.6781$\times 10^{-5}$ & 1.2440$\times 10^{-2}$ & 1.6098$\times 10^{-15}$ & 2.3870$\times 10^{-15}$ \\ \hline\\[-2mm]
\end{tabular}
\caption{$L^\infty$ error for the function $\Phi(x,y)=\sqrt{1+x^2+y^2}$. \label{tbl:1} }
\end{table}

\begin{table}[!h] 
\centering\setstretch{1.2}
 \begin{tabular}{ ccccc }
\hline
$h$ & Test 1 & Test 2 & Test 3 & Test 4 \\ \hline
$2^{-2}$&2.2524$\times 10^{-2}$ & 2.7012$\times 10^{-2}$ & 6.3363$\times 10^{-11}$ & 7.3175$\times 10^{-5}$ \\ \hline
$2^{-3}$&4.1574$\times 10^{-3}$ & 2.6801$\times 10^{-2}$ & 1.6653$\times 10^{-16}$ & 1.0270$\times 10^{-15}$ \\ \hline
$2^{-4}$&1.1233$\times 10^{-3}$ & 2.2223$\times 10^{-2}$ & 5.5511$\times 10^{-17}$ & 3.1919$\times 10^{-15}$ \\ \hline
$2^{-5}$&3.1368$\times 10^{-4}$ & 1.6967$\times 10^{-2}$ & 4.7184$\times 10^{-16}$ & 6.5781$\times 10^{-15}$ \\ \hline
$2^{-6}$&1.3201$\times 10^{-4}$ & 1.2500$\times 10^{-2}$ & 2.1649$\times 10^{-15}$ & 1.0464$\times 10^{-14}$ \\ \hline\\[-2mm]
\end{tabular}
\caption{$L^\infty$ error for the function $\Phi(x,y)=|x|+|y|$. \label{tbl:2} }
\end{table}

\st{ \begin{table}[!h] 
\centering\setstretch{1.2}
 \begin{tabular}{ccccc}
\hline
$h$ & Test 1 & Test 2 & Test 3 & Test 4 \\ \hline
$2^{-2}$&3.9093$\times 10^{-3}$ & 2.5104$\times 10^{-2}$ & 1.3878$\times 10^{-16}$ & 5.0248$\times 10^{-5}$ \\ \hline
$2^{-3}$&1.0340$\times 10^{-3}$ & 2.6475$\times 10^{-2}$ & 2.2204$\times 10^{-16}$ & 1.9706$\times 10^{-15}$ \\ \hline
$2^{-4}$&2.6643$\times 10^{-4}$ & 2.2113$\times 10^{-2}$ & 1.1102$\times 10^{-16}$ & 1.2323$\times 10^{-14}$ \\ \hline
$2^{-5}$&6.6964$\times 10^{-5}$ & 1.6920$\times 10^{-2}$ & 5.5511$\times 10^{-16}$ & 6.0507$\times 10^{-15}$ \\ \hline
$2^{-6}$&1.6781$\times 10^{-5}$ & 1.2440$\times 10^{-2}$ & 8.7846$\times 10^{-15}$ & 1.0159$\times 10^{-14}$ \\ \hline\\[-2mm]
\end{tabular}
\caption{$L^\infty$ error for the function $\Phi(x,y)=\sqrt{x^2+y^2}$. \label{tbl:3}}
\end{table}
}
We conclude this section with numerical experiments for the Dirichlet problem  
for $\det \nabla^2 u =\nu$ with $\nu=  \pi/2 \, \delta_{(1/4,1/2)} +  \pi/2 \, 
\delta_{(3/4,1/2)}$ using the standard finite difference discretization and the monotone scheme. The 
exact solution, taken from \cite{Benamou2014b}, is given by
\[ u(x,y) = \left\{ 
  \begin{array}{l l}
    |y-\frac{1}{2}| & \quad \text{if $\frac{1}{4} < x < \frac{3}{4}$}\\
    \min\bigg\{ \, \sqrt{(x-\frac{1}{4})^2 + (y-\frac{1}{2})^2}, \sqrt{(x-\frac{3}{4})^2 + (y-\frac{1}{2})^2}  \, \bigg\} & \quad \text{otherwise}.
  \end{array} \right.\]

Our results are reported on \textsc{Tables} \ref{tbl:4}, \ref{tbl:5} and  \textsc{Figure} \ref{fig:fig2}.
This example was chosen because many existing methods fail to capture the solution. 
\begin{table}
\begin{tabular}{ccccc} 
 \multicolumn{5}{c}{$h$}\\
%     $1/2^2$ &  $1/2^3$ &  $1/2^4$&  $1/2^5$ &  $1/2^6$ & $1/2^7$ & $1/2^8$ \\ \hline
%3.9093   $10^{-3}$ & 1.0340 $10^{-3}$ & 2.6643 $10^{-4}$ & 6.6964 $10^{-5}$ & 1.6781 $10^{-5}$ & 4.1965 $10^{-6}$ & 1.0492 $10^{-6}$ \\ 
    $1/2^2$ &  $1/2^3$&  $1/2^4$ &  $1/2^5$ &  $1/2^6$ \\ \hline
%& $1/2^7$ & $1/2^8$ \\ \hline
   1.40 $10^{-2}$ & 1.44 $10^{-2}$ & 1.31 $10^{-2}$ & 1.25 $10^{-2}$ & 1.17 $10^{-2}$  \\% & 4.1965 $10^{-6}$ & 1.0492 $10^{-6}$ \\ 
% &  &  &  &  & \\
\end{tabular}
\caption{Dirac masses. $\Phi(x,y)=\sqrt{x^2+y^2}$, standard discretization of the determinant of the Hessian with discrete local convexity. \label{tbl:4} }\label{}
\end{table}

\begin{table}
\begin{tabular}{ccccc} 
 \multicolumn{5}{c}{$h$}\\
%     $1/2^2$ &  $1/2^3$ &  $1/2^4$&  $1/2^5$ &  $1/2^6$ & $1/2^7$ & $1/2^8$ \\ \hline
%3.9093   $10^{-3}$ & 1.0340 $10^{-3}$ & 2.6643 $10^{-4}$ & 6.6964 $10^{-5}$ & 1.6781 $10^{-5}$ & 4.1965 $10^{-6}$ & 1.0492 $10^{-6}$ \\ 
    $1/2^2$ &  $1/2^3$&  $1/2^4$ &  $1/2^5$ &  $1/2^6$ \\ \hline
%& $1/2^7$ & $1/2^8$ \\ \hline
   4.31 $10^{-3}$ & 1.08 $10^{-3}$ & 2.70 $10^{-4}$ & 6.74 $10^{-5}$ & 1.68 $10^{-5}$  \\% & 4.1965 $10^{-6}$ & 1.0492 $10^{-6}$ \\ 
% &  &  &  &  & \\
\end{tabular}
\caption{Dirac masses. $\Phi(x,y)=x^2+y^2$, monotone discretization of the determinant of the Hessian with wide stencil convexity on a 9 point scheme. \label{tbl:5} }\label{}
\end{table}

\begin{figure}
\subfloat[Analytical 
solution.]{\includegraphics[width=.40\linewidth]{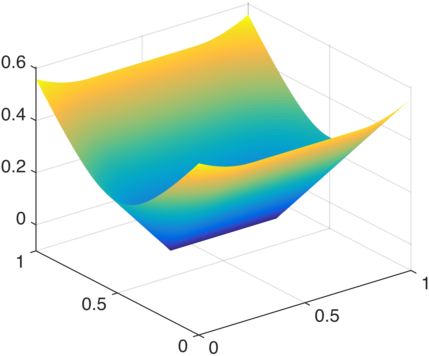}}\quad
\subfloat[Numerical 
solution.]{\includegraphics[width=.40\linewidth]{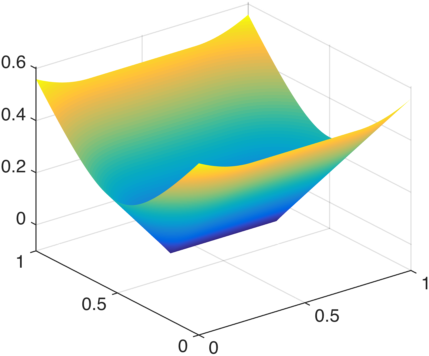}}
\caption{Dirac masses. The numerical solution is computed on a $65\times 65$ rectangular grid 
using the standard finite difference discretization. \label{fig:fig2}}
\end{figure}

%\section{Concluding Remarks}

\begin{rem}
Our numerical method also provides a technique for computing the convex envelope of boundary data and the convex envelope of a function defined on $\tir{\Omega}$. We recall from \cite{Lions86} that given $g$ on $\partial 
\Omega$ (satisfying the assumptions of this paper), the convex envelope of $g$ on  $\tir{\Omega}$ is the mimimum of $J$ over
$$
S = \{ v \in C(\tir{\Omega}),  \quad v=g \ \text{on} \ \partial \Omega,
\ v \ \text{convex on} \ \Omega \}.
$$
Also given any function $g$ defined on  $\tir{\Omega}$, the convex envelope of $g$ is the minimum of $J$ over
$$
S = \{ v \in C(\tir{\Omega}),  \quad v \leq g \ \text{on} \ \partial \Omega,
\ v \ \text{convex on} \ \Omega \}.
$$
\end{rem}

%\bibliographystyle{abbrv}
%\bibliography{OptBased}

\end{document}